\documentclass[12pt]{amsart}
\usepackage{amssymb,amsmath,amsthm,latexsym}
\usepackage{mathrsfs}
\usepackage{a4wide}
\usepackage[usenames,dvipsnames]{color}
\usepackage{euscript}
\usepackage{graphicx}
\usepackage{comment}
\usepackage{mdwlist,tikz-cd}
\usepackage{mathtools,dsfont,wasysym}

\usepackage{xargs}
\usepackage[disable, colorinlistoftodos,prependcaption,textsize=tiny,textwidth=3.5cm]{todonotes}
\presetkeys{todonotes}{fancyline}{}
\newcommandx{\todoin}[2][1=]{\todo[inline, caption={todo}, #1]{%
    \begin{minipage}{\textwidth-20pt}#2\end{minipage}}}

\newcommandx{\richard}[2][1=]{\todo[linecolor=blue,backgroundcolor=blue!25,bordercolor=blue,#1]{#2}}
\newcommandx{\richardin}[2][1=]{\richard[inline, caption={Richard}, #1]{%
    \begin{minipage}{\textwidth-30pt}#2\end{minipage}}}
\newcommandx{\tomek}[2][1=]{\todo[linecolor=red,backgroundcolor=red!25,bordercolor=red,#1]{#2}}
\newcommandx{\tomekin}[2][1=]{\tomek[inline, caption={Tomasz}, #1]{%
    \begin{minipage}{\textwidth-25pt}#2\end{minipage}}}
\newcommandx{\christian}[2][1=]{\todo[linecolor=red,backgroundcolor=yellow!25,bordercolor=yellow,#1]{#2}}
\newcommandx{\christianin}[2][1=]{\christian[inline, caption={Christian}, #1]{%
    \begin{minipage}{\textwidth-25pt}#2\end{minipage}}}

\DeclareFontFamily{U}{mathx}{\hyphenchar\font45}
\DeclareFontShape{U}{mathx}{m}{n}{
      <5> <6> <7> <8> <9> <10>
      <10.95> <12> <14.4> <17.28> <20.74> <24.88>
      mathx10
      }{}

\newcommand{\nn}[1]{{\vert\kern-0.25ex\vert\kern-0.25ex\vert #1 
    \vert\kern-0.25ex\vert\kern-0.25ex\vert}}

\newtheorem{theorem}{Theorem}
\newtheorem{lemma}[theorem]{Lemma}

\theoremstyle{remark}

\newtheorem*{remark*}{Remark}
\theoremstyle{definition}
\newtheorem{definition}[theorem]{Definition}

\numberwithin{equation}{section}

\newcounter{maintheorem}

\newtheorem{mainth}[maintheorem]{Theorem}
\newtheorem*{mainthprime*}{Theorem A$^\prime$}
\makeatother

\newcounter{smallromans}

\newenvironment{romanenumerate}
{\begin{list}{{\normalfont\textrm{(\roman{smallromans})}}}%
  {\usecounter{smallromans}\setlength{\itemindent}{0cm}%
   \setlength{\leftmargin}{5.5ex}\setlength{\labelwidth}{5.5ex}%
   \setlength{\topsep}{.5ex}\setlength{\partopsep}{.5ex}%
   \setlength{\itemsep}{0.1ex}}}%
{\end{list}}

\newcounter{smallromansdash}

{\end{list}}

\newcounter{bigromans} 
  {\end{list}}

\begin{document}


\baselineskip=17pt



\title[The existence of UFO implies projectively universal morphisms]{The existence of UFO implies\\ projectively universal morphisms}
\author[M.~Balcerzak]{Marek Balcerzak}
\address[M.~Balcerzak]{Institute of Mathematics, Lodz University of Technology, al. Politechniki~8, 93-590 {\L}\'od\'z, Poland}
\email{marek.balcerzak@p.lodz.pl}
\author[T.~Kania]{Tomasz Kania}
\address[T.~Kania]{Mathematical Institute\\Czech Academy of Sciences\\\v Zitn\'a 25 \\115 67 Praha~1\\Czech Republic  and  Institute of Mathematics and Computer Science\\ Jagiellonian University\\ {\L}ojasiewicza 6, 30-348 Krak\'{o}w, Poland
}
\email{tomasz.marcin.kania@gmail.com, kania@math.cas.cz}
\thanks{The second-named author acknowledges with thanks funding received from NCN project SONATA 15 No. 2019/35/D/\-ST1/\-01734.}

\date{\today}

\begin{abstract}
    Let $\mathcal C$ be a concrete category. We prove that if $\mathcal{C}$ admits a universally free object $\mathsf F$, then there is a projectively universal morphism $u\colon \mathsf F\to \mathsf F$, i.e., a morphism $u$ such that for any $B\in \mathcal{C}$ and $\tau\in {\rm Mor}(B)$ there exists an epimorphism $\pi\in {\rm Mor}(\mathsf F, B)$ such that $\pi \tau = u \pi$. This builds upon and extends various ideas by Darji and Matheron (\emph{Proc.~Am.~Math.~Soc.} 145 (2017)) who  proved such a result for the category of separable Banach spaces with contractive operators as well as certain classes of dynamical systems on compact metric spaces. Specialising from our abstract setting, we conclude that the result applies to various categories of Banach spaces/lattices/algebras, C*-algebras, etc.  \end{abstract}


\maketitle
\section{Introduction and the main results}



Darji and Matheron \cite{DarjiMatheron} proved that there exists a bounded linear operator $U\colon \ell_1\to \ell_1$ such that every bounded linear operator $T\colon E\to E$ on a~separable Banach space $E$ is a~factor of $\rho\cdot U$ ($\rho\geqslant \|T\|$), i.e., there is a contractive surjective linear operator $\pi \colon \ell_1\to E$ for which the following diagram commutes:

\begin{figure}[ht]
\centering
    \begin{tikzcd}[sep=huge]
    \ell_1 \arrow[dr,phantom, description, sloped]
      \arrow[d, "\pi"'] \arrow[r, "\rho\cdot U"] & \ell_1 \arrow[d, "\pi"] \\
    E \arrow[r, "T"] & E
    \end{tikzcd}
\end{figure}
The aim of the present note is to show that this is an artefact of a purely homological phenomenon that $\ell_1$ is a free object in the category of Banach spaces with contractive operators as morphisms (this is sometimes referred to as K\"othe's theorem, which says in addition that the space $\ell_1(\Gamma)$ for any set $\Gamma$ are the only free/projective objects in this category); here the forgetful functor $\Lambda$, required to properly interpret freeness, takes a~Banach space $E$ to its unit ball $B_E$ and restricts the operators (morphisms) accordingly. In particular, it is not a coincidence that we may take $\rho = 1$ in this setting and for general, possibly non-contractive maps we have to adjust the map to $\rho^{-1}T$, where $\rho \geqslant \|T\|$.\smallskip

Having recognised that, we prove that such a factorisation is not specific to `countably determined objects' and works in full generality way beyond the realm of Banach spaces. For this, we require to build an appropriate category-theoretic framework (the relevant terminology is explained at the end of this section).\smallskip
\begin{definition}
    We call a free object $\mathsf{F} = F(X)$ in $\mathcal{C}$, where $X\in \mathsf{Set}$, \emph{universally free} (UFO), whenever for every $B\in \mathcal{C}$, there exists an injection $\iota\colon X\to \Lambda(B)$ such that the free extension $\beta \iota \colon \mathsf F \to B$ is an epimorphism. 
\end{definition}
For example, in the category $\mathsf{Grp}_\lambda$ of all groups of cardinality at most $\lambda$, where $\lambda$ is an infinite cardinal, the free group on $\lambda$-many generators is a UFO. Similarly, in the category $\mathsf{Ab}_\lambda$ of Abelian groups of cardinality not exceeding $\lambda$, $\mathbb Z^{(\lambda)}$, the direct sum of $\lambda$-many copies of the group of integers, is a UFO.\smallskip

We openly admit that the notion of a UFO, as defined above, is not properly aligned with the spirit of Category Theory as it tacitly puts cardinality constraints on the objects when interpreted in the category of sets. Perhaps a more familiar notion is the one of a~\emph{projective generator}: when present in a category with coproducts, every object is a~target of an~epimorphism from the coproduct of a certain (possibly infinite) number of copies of the projective generator. This is still related to UFOs, however we want to apply Theorem~\ref{Thm:A} to objects naturally appearing in Analysis that are often grouped by their density or some other cardinal invariant. Consequently, if we were to replace UFOs by requiring suitable generators to exist, we would still have to introduce certain `cardinal scalings' on the underlying categories. Let us now state our first main result.

\begin{mainth}\label{Thm:A}
Let $(\mathcal{C}, \Lambda)$ be a concrete category admitting a UFO $\mathsf F$. Then there exists a~projectively universal morphism $u\in {\rm Mor}(\mathsf F)$, that is, for every $B\in \mathcal C$ and $\tau\in {\rm Mor}(B)$ there exists an epimorphism $\pi\colon \mathsf F \to B$ such that the following diagram commutes:
\begin{figure}[ht]
\centering
    \begin{tikzcd}[sep=huge]
    \mathsf F \arrow[dr,phantom, description, sloped]
      \arrow[d, "\pi"'] \arrow[r, "u"] & \mathsf F \arrow[d, "\pi"] \\
    B \arrow[r, "\tau"] & B.
    \end{tikzcd}
\end{figure}
\end{mainth}

Our motivation for considering the problem stems from the desire of extending the result due to Darji and Matheron to naturally appearing categories enriching the category of Banach spaces. Let us then list a number of categories that are of particular of interest to us to which Theorem~\ref{Thm:A} applies.
\begin{mainth}\label{Thm:B}
    Let $\lambda$ be an infinite cardinal. The following categories admit UFOs:
    \begin{romanenumerate}
        \item\label{i1} (involutive/commutative) semigroups/monoids of cardinality at most $\lambda$ with (involutive) semigroup homomorphisms;
        \item\label{i3} groups/Abelian groups of cardinality at most $\lambda$ with group homomorphisms;
        \item\label{i4} (semi)lattices/Boolean algebras of cardinality at most $\lambda$ with lattice homomorphisms;
        \item\label{i5} compact Hausdorff spaces of density at most $\lambda$ with continuous maps;
        \item\label{i6} Banach spaces of density at most $\lambda$ with contractive linear maps;
        \item\label{i7} Banach lattices of density at most $\lambda$ with contractive lattice homomorphisms;
        \item\label{i8} Banach algebras of density at most $\lambda$ with contractive algebra homomorphisms;
        \item\label{i9} Banach *-algebras of density at most $\lambda$ with contractive *-homomorphisms;
        \item\label{i10} $C^*$-algebras of density at most $\lambda$ with *-homomorphisms;
    \end{romanenumerate}
\end{mainth}
Somewhat surprisingly, in the category of Fr\'echet spaces, free objects are necessarily finite-dimensional (\cite{gejler1978extending}) so there are no UFOs therein.\smallskip

The proof of Theorem~\ref{Thm:B} is, in fact, a compilation of known, yet scattered in the literature, facts that we shall then only outline.
\begin{proof}[Proof of Theorem~\ref{Thm:B}]As for \eqref{i1}--\eqref{i5}; these are standard: let $\lambda$ be an infinite cardinal. The categories in scope comprise algebraic structures and a such they admit free objects, and in our setting, actual UFOs. The UFOs may be sometimes concretely identified: 
\begin{itemize}
    \item for semigroups, monoids, groups: the free semigroup $S_\lambda$, the free monoid $S^\#_\lambda$, and the free group ${\rm Fr}_\lambda$ on $\lambda$ generators (all with concatenation);
    \item for commutative semigroups, monoids, and groups: $\mathbb N^{(\lambda)}$, $\mathbb N_0^{(\lambda)}$, and $\mathbb Z^{(\lambda)}$ (the direct sum of $\lambda$-many copies of the semigroups $\mathbb N_0$, $\mathbb N$, and $\mathbb Z$), respectively;
    \item for (semi)lattices: the set of all finite subsets of a set of cardinality $\lambda$ with the union and intersection as (semi)lattice operations; 
    \item Boolean algebras: the algebra of clopen subsets of the Cantor cube $\{0,1\}^\lambda$.
\end{itemize}
For \eqref{i5}, let $K$ be a compact Hausdroff space of density $\lambda$ and $D\subset K$ be a dense subset of cardinality $\lambda$. The identity map $\iota \colon D\to D$, where in domain we consider the discrete topology is continuous. Consequently, the \v{C}ech--Stone extension $\beta\iota\colon \beta D\to K$ is a continuous surjection. This means that $C(\beta \lambda)$, the \v{C}ech--Stone compactification of the discrete space of cardinality $\lambda$, is the sought UFO.

For \eqref{i6}, that $\ell_1(\lambda)$ is a UFO follows essentially from K\"othe's theorem; see also \cite[Lemma 1.4]{castillo1997three}. In the case of Banach lattices \eqref{i7}, it follows from \cite[Corollary 2.9]{aviles2018free} that the free Banach lattice over the Banach space $\ell_1(\lambda)$, ${\rm FBL}[\ell_1(\lambda)]$, is a UFO for Banach lattices of density at most $\lambda$.

Every (unital) Banach algebra of density $\lambda$ is a quotient of the semigroup convolution algebra $\ell_1(S_\lambda)$ (respectively, $\ell_1(S_\lambda^\#)$), where $S_\lambda$ (respectively, $S_\lambda^\#$) is the free semigroup (monoid) on $\lambda$ generators. In the Abelian case, one may consider $\ell_1(\mathbb N^{(\lambda)})$ and $\ell_1(\mathbb N^{(\lambda)}_0)$, so that the four listed algebras are UFOs in the respective categories. This is explained in more detail in \cite[Lemma 2.8]{horvath2021unital}. In the case of Banach *-algebras, the respective objects are $\ell_1(S_\lambda^*)$ and $\ell_1(S_\lambda^{*\#})$, where $S_\lambda^*$ and $S_\lambda^{*\#}$ are, respectively, the free involutive semigroup and free involutive monoid on $\lambda$ generators.

For \eqref{i10}, the full group $C^*$-algebra $C^*_{\max}({\rm Fr}_\lambda)$ on the free group on $\lambda$ generators is an UFO in the respective categories. This follows from standard fact that the the unitary group of a $C^*$-algebra of density $\leqslant \lambda$ (von Neumann algebra acting on a Hilbert space of density $\leqslant \lambda$) contains a dense subgroup of cardinality $\lambda$ that, as such is a quotient of ${\rm Fr}_\lambda$ that extends to either group algebra; the by maximaility of $C^*_{\max}({\rm Fr}_\lambda)$, the group epimorphism extends to a *-epimorphism on the level of $C^*$-algebras. \end{proof}
The list from Theorem~\ref{Thm:B} is by no means complete; it only reflects the authors' personal preferences. The reader is invited to explore the UFO phenomenon in their concrete categories of interest.
\subsection*{Notation and terminology}In the present paper we follow standard conventions. For a given category $\mathcal{C}$ and objects $B,C\in \mathcal{C}$ we denote by ${\rm Mor}(B,C)$ the class of morphisms from $B$ to $C$; we write ${\rm Mor}(B)$ for ${\rm Mor}(B,B)$.

Let $\mathcal{C}$ be a concrete category, that is, a category with a fixed faithful (\emph{forgetful}) functor $\Lambda\colon \mathcal{C}\to \mathsf{Set}$ to the category of all sets; formally a concrete category is then the pair $(\mathcal{C}, \Lambda)$. An object $F\in \mathcal{C}$ is \emph{free} (on $X\in \mathsf{Set}$), when it comes equipped with a morphism $\eta_X\colon X\to \Lambda(F)$ (in $\mathsf{Set}$) such that for every $B\in \mathcal{C}$ and a function $f\colon X\to \Lambda(B)$ there exists a unique morphism $\beta f\in {\rm Mor}(F, B)$ such that $\Lambda(\beta f)\circ \eta_X = f$. 

\newpage

\section{Proof of Theorem~\ref{Thm:A}}

In order to prove the main result we need to extend \cite[Fact 3.2]{DarjiMatheron} beyond countable sets.

\begin{lemma}\label{lem:fact}
Let $X$ be an infinite set. Then there exists an injective map $\mu\colon X\to X$ such that for any other injective map 
$\sigma\colon X\to X$ there are a set $A\subseteq X$ and a bijection $\pi_A\colon X\to A$ such that 
$\mu[A]\subseteq A$ and $\sigma=\pi_A^{-1}\mu\pi_A$.
\end{lemma}
\begin{proof}
Let $\lambda = |X|$. For any injective map $\sigma\colon X\to X$, let $G_\sigma$ denote the digraph (directed graph) on $X$ induced by $\sigma$: $\overrightarrow{ij}$ is a directed edge of $G_\sigma$ iff $\sigma(i)=j$. We then have only three possible types of components of $G_\sigma$:
\begin{itemize}
    \item[$1^\circ$] a cycle of length $n\geqslant 2$ for some $n\in\mathbb N$;
    \item[$2^\circ$] $\{i_0,i_1,i_2,\dots\}$ where $i_k$ ($k\in\mathbb N$) are distinct, $i_0\notin\sigma[X]$, and the only directed edges are of the form $\overrightarrow{i_k i_{k+1}}$;
    \item[$3^\circ$] $\{ \dots ,i_{-1},i_0, i_1,\dots\}$ where $i_k$ ($k\in\mathbb N$) are distinct and the only directed edges are of the form $\overrightarrow{i_k i_{k+1}}$.
\end{itemize} 
Clearly, the image $\sigma[X]$ is the union of these components. Note the converse situation: if $G$ is a digraph on $X$, it defines an injection $\sigma\colon X\to X$ such that $G=G_\sigma$ with $\sigma[{X}]$ equal to the union of all components of $G$.

Now, let us construct a kind of universal digraph.
Consider two complementary subsets $X_1$ and $X_2$ of $X$, both of size $\lambda$. Then, let us consider a~partition of $X_1$ into sets of three types: 
\begin{itemize}
    \item $V_{n,\alpha}$ ($n\in\mathbb N$, $n\geqslant 2$, $\alpha<\lambda$) with $|V_{n,\alpha}|=n$;
    \item $Y_{t,\alpha}$ ($t\in X_2$, $\alpha<\lambda$), where 
    $Y_{t,\alpha}=\{x_{0,\alpha}, x_{1,\alpha},x_{2,\alpha},\dots\}$ with distinct elements;
    \item $Z_\alpha$ (for $\alpha<\lambda$), where $Z_\alpha=\{\dots, z_{-1,\alpha}, z_{0,\alpha}, z_{1,\alpha},\dots\}$ with distinct elements.
\end{itemize}
Such a partition exists since $\lambda\cdot\omega=\lambda$.
Each of these sets can be treated as a component of a~digraph $G$ of type $1^\circ$, $2^\circ$, or $3^\circ$. Then $G=G_\mu$ with the respective injective map $\mu\colon X\to X$. 

Now, let $\sigma\colon X\to X$ be an injective map and consider all components
of $G_\sigma$.
Then we find a copy of these components among all components of $G_\mu$. Let all vertices of this copy form a set $A$, and let $\pi_A\colon X\to A$ be the respective correspondence (bijection) realising this copy. Observe that $A$ and $\pi_A$ are as desired.
\end{proof}

We are now ready to prove Theorem~\ref{Thm:A}.
\begin{proof}[Proof of Theorem~\ref{Thm:A}]

Without loss of generality we may suppose that for all objects $B \in \mathcal C$ and morphisms $\psi\in {\rm Mor}(B)$, $\Lambda(\psi)$ has a fixed point $0_\phi$. (This can be achieved by the process of adjoining a new point to each $\Lambda(B)$ ($B\in \mathcal{C}$) and extending the corresponding $\mathsf{Set}$-morphisms accordingly.)

Suppose that $\mathsf F = F(X)$ is a UFO. Consider the injection $ \mu \colon X\to X$ from Lemma~\ref{lem:fact}. Let $\beta \mu \colon \mathsf{F} \to \mathsf{F}$ be the associated morphism. We \emph{claim} that $\beta \mu$ is the sought universal projective morphism.

Let $B\in \mathcal{C}$ and let $\tau \in {\rm Mor}(B)$. Pick $D_0\subseteq \Lambda(B)$ or cardinality $\lambda$ such that any surjection $h\colon X\to D_0$ extends to an epimorphism $\beta h\colon \mathsf F\to B$. Set $D = \bigcup_{n=1}^\infty D_n$, where $D_{n+1} = \Lambda(\tau)[D_n]$ ($n = 1, 2, \ldots$).

Let $\{z_x \colon x\in X\}$ be an $X$-indexed enumeration of $D$ in which every element $z\in D$ appears infinitely often. Consequently, we may find an injection $\sigma\colon X\to X$ such that $(\Lambda\tau) (z_x) = z_{\sigma(x)}$ ($x\in X$)\footnote{This does not force $\Lambda(\tau)$ to be injective on $D$ due to repetitions in the enumeration.}. Let $A$ and $\pi_A$ be chosen as in Lemma~\ref{lem:fact} applied to $D$. Set $\pi(x) = z_{\pi_A^{-1}(x)}$ if $x\in A$ and $\pi(x) = 0_\tau$ otherwise. As $\Lambda(\beta \pi)[\Lambda(\mathsf F)]$ contains $D$, $\beta \pi\colon \mathsf F \to B$ is an epimorphism. We \emph{claim} that $\tau$ is a factor of $\beta \mu$ as witnessed by $\pi$.

By the freeness of $\mathsf F$ with respect to $X$, it suffices to check that $\Lambda(\pi (\beta \mu))(x) = \Lambda (\tau \pi)(x)$ for $x\in X$. For this, observe that 

\[
    \Lambda(\pi (\beta \mu))(x) = \Lambda(\pi) (\Lambda(\beta \mu)(x)) = \Lambda(\pi) (\mu(x)) = \pi(\mu(x)) = \left\{\begin{array}{ll} z_{\pi^{-1}_A(\mu(x))}, & x\in A, \\ 0_\tau, & x\notin A.  \end{array} \right.
\]
On the other hand, if $x\notin A$, then 
\[
    \Lambda(\tau \pi)(x) = \Lambda(\tau)(\pi(x)) = \Lambda(\tau)(0_\tau) = \tau(0_\tau) = 0_\tau. 
\]
When $x\in A$, we have
\[
    \Lambda(\tau \pi)(x) = \Lambda(\tau)(\pi(x)) = \Lambda(\tau)(z_{\pi_A^{-1}(x)}) = z_{\sigma(\pi_A^{-1}(x))} = z_{\pi^{-1}_A(\mu(x))},
\]
where the latter equality follows directly from Lemma~\ref{lem:fact}. Consequently, $\Lambda(\pi (\beta \mu)) = \Lambda (\tau \pi)$ so $\pi (\beta \mu) = \tau \pi$.
\end{proof}

\bibliography{bibliography.bib}
\bibliographystyle{plain}

\end{document}